\documentclass[12pt]{article}
\usepackage{amssymb,amsmath,amsthm,url}
\usepackage[matrix,arrow]{xy}

\newcommand{\fm}{\mathfrak{m}}
\newcommand{\lbr}[2]{ [ \hspace*{-1.6pt} [ #1 , #2 ] \hspace*{-1.7pt} ] }

\newtheorem{lemma}{Lemma}
\newtheorem{cor}{Corollary}
\newtheorem{theorem}{Theorem}
\newtheorem{prop}{Proposition}
\theoremstyle{definition}
\newtheorem{remark}{Remark}
\newtheorem*{definition}{Definition}
\title{A cohomological construction of modules \\ over Fedosov deformation quantization algebra. \\ The global case}
\author{S. A. Pol'shin\thanks{E-mail: polshin.s at gmail.com}\\
{\small Institute for Theoretical Physics} \\
{\small NSC Kharkov Institute of Physics and Technology} \\
{\small Akademicheskaia St. 1, 61108 Kharkov, Ukraine }}
\date{}

\begin{document}
\maketitle

\begin{abstract}
Let $(M,\omega)$ be a symplectic manifold,  $\mathcal{D}\subset TM$ a real polarization on $M$ and $\wp$ a leaf of $\mathcal{D}$. We construct a Fedosov-type star-product $\ast_L$ on $M$ such that  $C^\infty (\wp)[[h]]$ has the natural structure of a left module over the deformed algebra $(C^\infty (M)[[h]], \ast_L)$.

MSC 2000: 53D55 53D50 17B55 53C12 13B30

Key words: Fedosov quantization, polarization,
contracting homotopy, projective module, maximal ideal.
\end{abstract}

\section{Introduction}

The ordinary deformation quantization scheme~\cite{Ber74,BFFLS} deals with the deformation of the point-wise product of functions on a symplectic manifold. However, it was realized in the early 90s that a purely algebraic approach based on appropriate geometric structures may be more efficient~\cite{Hueb,Grab}. The most successful attempt  in this direction was made by Fedosov~\cite{Fedosov1,Fedosov2} who also constructed star-products on an arbitrary symplectic manifold as a by-product; the algebraic nature of Fedosov's construction was shown by Donin~\cite{Donin} and Farkas~\cite{Farkas}.

The problem of constructing modules over Fedosov
deformation quantization which generalize the states of textbook
quantum mechanics  is of great interest (see~\cite{WaldStates,BordTrav} for a review). In a recent paper~\cite{IJGMP} this problem has been solved in a certain neighborhood $U$ of an arbitrary point of a symplectic manifold $M$. In the present paper we extend this result onto the whole $M$. The main technical difficulty of this generalization comes from the fact that $\Gamma TM$ is projective as $C^\infty (M)$-module in general, while $\Gamma (U,TM)$ is free. To circumvent this difficulty, we systematically use the localization w.r.t. the maximal ideals of $C^\infty (M)$ and thus reduce the projective case to the free one. This allows us to construct adapted star products on $M$ in the sense of~\cite{BGHHW}.

The plan of the present paper is the following. In Sec.~2 we construct the Weyl algebra for $\Gamma TM$ and prove an analog of the Poincar\'e-Birkhoff-Witt theorem,
in Sec.~3 we consider the Koszul complex, in Sec.~4 we define various ideals associated with a polarization $\mathcal{D}\subset TM$, in Sec.5 we introduce the symplectic connection on $M$ adapted to $\mathcal{D}$ and study its properties w.r.t. the ideals, in Sec.~6 we define the Fedosov complex and prove the main result.

\section{Weyl algebra}

Let $M$ be a symplectic manifold, $\dim M=2N$, let $A=C^\infty(M,\mathbb{R})$ be an $\mathbb{R}$-algebra of smooth functions on $M$ with pointwise multiplication, and $E=\Gamma TM$ the set of all smooth  vector fields on $M$  with the natural
structure of an unitary $A$-module. By $T(E)$ and $S(E)$ denote
the tensor and symmetric algebra of the $A$-module $E$ respectively,
and let $\wedge E^\ast$ be the algebra of smooth differential forms
on $M$. Let $\omega \in \wedge^2 E^\ast$ be a symplectic form on
$M$ and let $u:\ E\rightarrow \wedge^1 E^\ast$ be the mapping
$u(x)y=\omega (x,y), \ x,y\in E$. All the tensor products of modules in the present paper will be taken over $A$ unless otherwise indicated.

\begin{theorem}[Serre-Swan]
There is an equivalence of the category of vector bundles over $M$ with the category of finitely generated projective $A$-modules.
\end{theorem}

For different variants and generalizations of Serre-Swan theorem see~\cite{Morye,Vaser,Vin} and references therein.

\begin{cor}[\cite{Rin},pp.202-3]\label{swan}
As an $A$-module, $E$ is a finitely generated projective $A$-module.
\end{cor}
Let $\lambda$
be an independent variable (physically $\lambda=-i\hbar$) and
$A[\lambda]=A\otimes_\mathbb{R} \mathbb{R}[\lambda]$ etc. In the
sequel we will write $A,E$ etc. instead of $A[\lambda],
A[[\lambda]],E[\lambda],$ $E[[\lambda]]$ etc. Let $\mathcal{I}_W$
be the two-sided ideal in $T(E)$ generated by the relations $x \otimes
y-y \otimes x-\lambda\omega (x,y)=0$. The factor-algebra
$W(E)=T(E)/\mathcal{I}_W$ is called the Weyl algebra of $E$, so
we have the short exact sequence of $A$-modules
\begin{equation}\label{shes-1}
\xymatrix@1{
0 \ar[r] & \mathcal{I}_W \ar[r] &
 T(E) \ar[r] & W(E) \ar[r] & 0
}
\end{equation}
and let $\circ$ be the multiplication in $W(E)$.

An $N$-dimensional real distribution $\mathcal{D}\subset TM$ is
called a polarization   if it is (a) lagrangian, i.e.
$\omega(x,y)=0$ for all $x,y\in \mathcal{D}$ and (b) involutive,
i.e. $[x,y]\in \mathcal{D}$ for all $x,y\in \mathcal{D}$, where
$[.,.]$ is the commutator of vector fields on $M$. It is well
known~\cite{Etayo} that  we can always choose a lagrangian
distribution $\mathcal{D}'$ transversal to $\mathcal{D}$ and let
$L$ and $L'$ be the $A$-modules of smooth  vector fields
on $M$ tangent to $\mathcal{D}$ and $\mathcal{D}'$ respectively,
then $E=L\oplus L'$.

\begin{theorem}\label{S-W-isom}
There exists an $A$-module isomorphism $\pi:\ \xymatrix@1{S(E) \ar[r]^\cong & W(E)}$.
\end{theorem}
\begin{proof}
Let $\fm\in\mathop{\mathrm{Specm}}A$ be a maximal ideal in $A$. For an arbitrary $A$-module $P$ consider its localization  $P\rightarrow P_{\fm}=A_{\fm}\otimes P$. It is well known that $(P\otimes Q)_{\fm} =P_{\fm} \otimes_{A_{\fm}} Q_{\fm}$, so $(T(E))_{\fm}= T(E_{\fm})$.

Due to Corollary~\ref{swan}, $E$ is flat and finitely presentable as the $A$-module, so there exists an isomorphism of $A_{\fm}$-modules
$$(E^*)_{\fm}\cong (E_{\fm})^* := \mathop{\mathrm{Hom}}_{A_{\fm}} (E_{\fm},A_{\fm})$$
and it may be extended to an isomorphism $(\wedge E^*)_{\fm} \cong \wedge E^*_{\fm}$. Let $\omega\in \wedge^2 E^*$ and $x,y\in E$, then there exists an element $\omega_{\fm}\in \wedge^2 E^*_{\fm}$ such that
\begin{equation}\label{o-o-m}
\omega_{\fm}(x/s,y/s')=\omega(x,y)/ss' \qquad \forall x/s,y/s'\in E_{\fm}
\end{equation}
as a result of the composition of the localization map and the mentioned isomorphism.

It is easily seen that $(\mathcal{I}_W)_{\fm}$ is the ideal in $T(E_{\fm})$ generated by the relations $x/1 \otimes_{A_{\fm}} y/1 - y/1 \otimes_{A_{\fm}} x/1 -\lambda\omega_{\fm}(x/1,y/1)=0$. Since the functor $A_{\fm}\otimes$ is exact, we have a short exact sequence of $A_{\fm}$-modules
$$\xymatrix@1{
0 \ar[r] & (\mathcal{I}_W)_{\fm} \ar[r] &
 T(E_{\fm}) \ar[r] & (W(E))_{\fm} \ar[r] & 0,
}
$$
so $W(E_{\fm})\cong (W(E))_{\fm}$, where $W(E_{\fm})$ is defined using the 2-form $\omega_{\fm}$ on $E_{\fm}$. Analogously $S(E_{\fm})\cong (S(E))_{\fm}$. Since $E_{\fm}$ is free as $A_{\fm}$-module and $E_{\fm}=L_{\fm}\oplus {L'}_{\fm}$, the theorem is proved using Prop.~\ref{PBW} below.
\end{proof}
\begin{remark}
For an arbitrary projective $A$-module $E$ Theorem~\ref{S-W-isom} was proved in~\cite{Rin} (see also~\cite{HKR,Hueb}). Here we gave a slightly different proof which is more appropriate for our purposes.
\end{remark}

\begin{remark}
It is well known that there is no 1-1 correspondence between "points" of $\mathop{\mathrm{Specm}}A$ and the points of $M$ unless $M$ is compact. So, $\omega_\fm$ need not be nondegenerate.
\end{remark}

Let $\alpha,\alpha_1,\ldots =1,\ldots,\nu$ and $\beta,\beta_1,\ldots=\nu+1,\ldots,\nu+\nu'$. Choose an $A_{\fm}$-basis $\{ e_i|\, i=1,\ldots,\nu+\nu' \}$ in $E_{\fm}$  such that $\{ e_\alpha|\,
\alpha=1,\ldots,\nu \}$ and $\{ e_\beta|\, \beta=\nu+1,\ldots,\nu+\nu'
\}$ are the bases in $L_{\fm}$ and ${L'}_{\fm}$ respectively.  Let $i_1,\ldots,i_p=1,\ldots,\nu+\nu'$ and
let $I=(i_1, \ldots,i_p)$  be an arbitrary sequence of indices. We
write $e_I=e_{i_1}\otimes_{A_\fm}\ldots\otimes_{A_\fm} e_{i_p}$ and we say that
the sequence $I$ is nonincreasing if $i_1 \geq i_2 \geq \ldots\geq
i_p$. We consider $\{\emptyset \}$ as a nonincreasing sequence and
$e_{ \{ \emptyset \} }=1$. We say that a sequence $I$ is of
$\alpha$-length $n$ if it contains $n$ elements less or equal than
$\nu$. Let $\Upsilon^n$ be the set of all nonincreasing sequences of
$\alpha$-length $n$ and $\Upsilon_n =\bigcup_{p=n}^\infty
\Upsilon^p$. The following proposition is a variant of the Poincare-Birkhoff-Witt theorem~\cite{Sridharan}.

\begin{prop}[Poincare-Birkhoff-Witt]\label{PBW}
Let $\tilde{S}(E_{\fm})$ be  the $A_{\fm}$-submodule of $T(E_{\fm})$ generated by $\{ e_I |\, I\in \Upsilon_0\}$. Then

(a) The restrictions $\mu_S |\tilde{S}(E_{\fm})$ and
$\mu_W |\tilde{S}(E_{\fm})$ of the
canonical homomorphisms $\mu_S :\, T(E_{\fm})\rightarrow
S(E_{\fm})$ and $\mu_W :\, T(E_{\fm})\rightarrow W(E_{\fm})$
are $A_{\fm}$-module isomorphisms.

(b) $\{ \mu_S (e_I) |\, I\in \Upsilon_0\}$ and
$\{ \mu_W (e_I) |\, I\in \Upsilon_0\}$ are $A_{\fm}$-bases of $S(E_{\fm})$ and $W(E_{\fm})$ respectively.

(c) $T(E_{\fm})=\tilde{S}(E_{\fm})\oplus (\mathcal{I}_W)_{\fm}$.
\end{prop}

\begin{prop}
Under the assumptions of Prop.~\ref{PBW}, the choice of  bases in  $L_{\fm}$ and ${L'}_{\fm}$ does not affect the resulting isomorphism $W(E_{\fm})\stackrel{\cong}{\rightarrow} S(E_{\fm})$.
\end{prop}
\begin{proof}
Let $\{ e'_{i}=A_{i}^j e_j\}$ be a new basis in $E_{\fm}$ such that
$A_{\alpha}^\beta=A^\alpha_{\beta}=0$  and let $\tilde{S}'(E_{\fm})$ be the submodule in $T(E_{\fm})$ generated by $ \{ e'_{I}|\, I\in \Upsilon_0\}$. Since both $L_{\fm}$ and ${L'}_{\fm}$ are isotropic w.r.t. $\omega_\fm$, we see that for any element $a'\in \tilde{S}'(E_{\fm})$ there exists an element $a\in \tilde{S}(E_{\fm})$  such that $\mu_W (a)=\mu_W (a')$ and $\mu_S (a)=\mu_S (a')$. Due to Prop.~\ref{PBW}(c) such an element is unique and the map $a'\mapsto a$ is an isomorphism.
\end{proof}

\section{Koszul complex}

Let
$$ a=x_1 \otimes \ldots \otimes x_m \otimes y_1 \wedge \ldots\wedge y_n \in
T^m (E)\otimes \wedge^n E^\ast.$$ Define the Koszul differential
of bidegree $(-1,1)$ on $T^\bullet (E)\otimes \wedge^\bullet
E^\ast$ as
$$\delta a=\sum\limits_i x_1 \otimes\ldots\otimes \hat{x}_i \otimes\ldots\otimes
x_m \otimes u(x_i) \wedge y_1 \ldots \wedge y_n .$$

 Since $E^*$ is   projective  and so  $\wedge E^*$  is, we see that the functor $\otimes\wedge E^*$ is exact and due to~(\ref{shes-1}) we have a short exact sequence of $A$-modules
\begin{equation}
\xymatrix@1{ 0 \ar[r] & \mathcal{I}_W \otimes \wedge E^*
\ar[r] & T(E) \otimes
\wedge E^* \ar[r] & W(E) \otimes
\wedge E^* \ar[r] & 0. }
\end{equation}
It is easily seen that $\delta$ preserves $\mathcal{I}_W \otimes
\wedge E^\ast$, so it induces a well-defined differential on
$W(E)\otimes\wedge E^\ast$. It is well known
that $u$ is an isomorphism due to the nondegeneracy of $\omega$. So
we can define the so-called contracting homotopy of bidegree
$(1,-1)$ on $S^\bullet (E)\otimes\wedge^\bullet E^\ast$ which to
an element
$$a=x_1 \odot\ldots\odot x_m \otimes y_1 \wedge\ldots\wedge y_n
\in S^m (E)\otimes \wedge^n E^*,$$
where $\odot$ is the multiplication in $S(E)$, assigns the element
$$\delta^{-1} a=\frac{1}{m+n}\sum\limits_i (-1)^{i-1} u^{-1} (y_i)\odot x_1 \odot\ldots\odot x_m \otimes y_1                                                                 \wedge\ldots\wedge \hat{y}_i
\wedge\ldots\wedge y_n $$
for $m+n>0$ and $\delta^{-1} a=0$ for $m=n=0$.

Let $a=\sum\limits_{m,n\geq 0} a_{mn}$, where $a_{mn}\in S^m (E)\otimes \wedge^n E^*$ and $\tau:\ a\mapsto a_{00}$ is the projection onto the component of bidegree $(0,0)$. Carry $\delta$ to $S(E)\otimes \wedge E^*$ using the canonical homomorphism $T(E) \otimes\wedge E^* \rightarrow S(E)\otimes \wedge E^*$.  Then it is well known that the following equality
\begin{equation}\label{eq1}
\delta\delta^{-1}+\delta^{-1}\delta+\tau=\mathop{Id}
\end{equation}
holds. Carry the grading of $S(E)$ to $W(E)$ using the
isomorphism $S(E)\cong W(E)$. Since localization is a homomorphism of graded modules and $W^1 (E_\fm)\cong E_\fm$, we see that $W^1 (E)\cong E$ and we will identify them.

\begin{prop}\label{d-comm}
$\delta$ commutes with the $A$-module isomorphism $\pi \otimes \mathop{\mathrm{Id}}$ from Theorem~\ref{S-W-isom}.
\end{prop}
\begin{proof}
$\omega_{\fm}$  induces the homomorphism $u_{\fm}:\ E_{\fm}\rightarrow E^*_{\fm}$ which makes the following diagram  commute:
\begin{equation}\label{u-delta}
    \xymatrix{
 E \ar[r]^u \ar[d] & E^* \ar[d] \\
E_{\fm} \ar[r]^{u_{\fm}} & E^*_{\fm}.
}
\end{equation}
(note that $u_\fm$ needs not be an isomorphism).
Then we can define the Koszul differential $\delta_\fm$ on $W(E_{\fm})\otimes_{A_{\fm}} \wedge E^*_{\fm}$
which commutes with the composition of the localization map and the isomorphism $(W(E)\otimes \wedge E^*)_\fm \cong W(E_\fm)\otimes_{A_\fm} \wedge E^*_\fm$.

Let $\iota_m \ (m=1,2)$ be the natural embedding of the $m$th direct summand in the rhs of Prop.~\ref{PBW}(c) into $T(E_{\fm})$, so
$\mu_{S,W}|\tilde{S}(E_{\fm})=\mu_{S,W}\iota_1$. Then
from Prop.~\ref{PBW}(c) it follows
that the short exact sequence of $A_{\fm}$-modules
$$\xymatrix@1{
0 \ar[r] & (\mathcal{I}_W)_{\fm} \ar[r]^{\iota_2} &
 T(E_{\fm}) \ar[r]^{\mu_W} & W(E_{\fm}) \ar[r] & 0
}
$$
splits, whence we have another short exact sequence of $A_{\fm}$-modules
\begin{equation}\label{split}
\xymatrix@1@C+9pt{ 0 \ar[r] & (\mathcal{I}_W)_{\fm} \otimes_{A_{\fm}} \wedge E^*_\fm \ar[r]^(.50){\iota_2 \otimes\mathop{{\rm id}}} &
T(E_{\fm}) \otimes_{A_\fm} \wedge E^*_{\fm}
\ar[r]^(.48){\mu_W \otimes \mathop{{\rm id}}} &
W(E_{\fm}) \otimes_{A_{\fm}} \wedge E^*_{\fm} \ar[r] & 0
}
\end{equation}
and $\iota_1\otimes\mathop{{\rm id}}$ is the natural embedding
of $\tilde{S}(E_{\fm})\otimes_{A_{\fm}}\wedge E^*_{\fm}$ into
$T(E_{\fm})\otimes_{A_{\fm}}\wedge E^*_{\fm}$.

It is easily seen that $\delta_\fm$ preserves
$\tilde{S}(E_{\fm})\otimes_{A_{\fm}} \wedge E^*_{\fm}$,  so each arrow of the following commutative diagram of $A_{\fm}$-modules commutes with $\delta_\fm$.
$$\xymatrix{
 & T(E_\fm)\otimes_{A_{\fm}} \wedge E^*_{\fm} \ar[ddl]_{\mu_S
\otimes \mathop{{\rm id}}} \ar[ddr]^{\mu_W \otimes\mathop{{\rm id}}} \\
 & \tilde{S}(E_{\fm})\otimes_{A_{\fm}} \wedge E^*_\fm \ar[u]|{\iota_1
\otimes \mathop{{\rm id}}} \ar[dl]^{\mu_S \iota_1 \otimes
 \mathop{{\rm id}}}|(.35)\cong \ar[dr]_{\mu_W \iota_1 \otimes
\mathop{{\rm id}}}|(.35)\cong \\ S(E_{\fm})\otimes_{A_{\fm}} \wedge E^*_{\fm}
&& W(E_\fm)\otimes_{A_{\fm}} \wedge E^*_\fm. }$$
Then $\delta_\fm$ commutes with the $A_{\fm}$-module isomorphism $\pi_\fm \otimes \mathop{{\rm id}}:=\mu_W \iota_1 (\mu_S \iota_1)^{-1}\otimes \mathop{{\rm id}}$. Due to the construction of $\pi$ we have $\left( (\pi\otimes \mathop{{\rm id}})\delta a-\delta (\pi\otimes \mathop{{\rm id}}) a \right)_\fm= (\pi_\fm \otimes \mathop{{\rm id}})\delta_\fm a_\fm-\delta_\fm (\pi_\fm \otimes \mathop{{\rm id}}) a_\fm$ for all $a\in S(E)\otimes \wedge E^*$ and $\fm\in \mathop{\mathrm{Specm}}A$. So, $(\pi\otimes \mathop{{\rm id}})\delta =\delta (\pi\otimes \mathop{{\rm id}})$, which proves the proposition.
\end{proof}

Carry the contracting homotopy $\delta^{-1}$ and the projection $\tau$ from $S(E)\otimes\wedge E^*$ to $W(E)\otimes\wedge E^*$ via the
isomorphism of Theorem~\ref{S-W-isom}, then the equality~(\ref{eq1}) remains true due to Prop.~\ref{d-comm}.
Let $\delta W^\bullet =(W(E)\otimes \wedge^n E^*, \ \delta)$, then
from~(\ref{eq1}) it follows that
\begin{equation}\label{eq4} H^0 (\delta W^\bullet)=A, \qquad H^n
(\delta W^\bullet)=0, \  n>0.
\end{equation}

\section{The ideals}

Let $\mathcal{I}_\wedge$ be the ideal in $\wedge E^*$ whose elements annihilate the polarization $L$, i.e. $\mathcal{I}_\wedge=\sum\nolimits_{n=1}^\infty
\mathcal{I}_\wedge^n$, where
$$\mathcal{I}_\wedge^n =\{ \alpha\in \wedge^n E^* |\, \alpha (x_1,\ldots,x_n)=0
\ \forall x_1,\ldots,x_n \in L \}.$$
It is well known that locally (i.e. in a certain neighborhood of an arbitrary point of $M$) $\mathcal{I}_\wedge$ is generated by $N$ independent 1-forms which are the basis of $\mathcal{I}_\wedge^1$. On the other hand, $L$ is lagrangian, so for the dimensional reasons we obtain $u(L)=\mathcal{I}_\wedge^1$, so
\begin{equation}\label{eq5}
\mathcal{I}_\wedge=(u(L)).
\end{equation}

Let $\mathcal{I}_L$ be the left ideal in $W(E)$ generated by the elements of $L$. Since  $\wedge E^*$ is projective, we have an  injection $\mathcal{I}_L\otimes\wedge E^*\hookrightarrow W(E)\otimes\wedge E^*$.

Consider $L^*$ as the submodule in $E^*$ whose elements annihilate $L'$. Then considering a certain neighborhood of an arbitrary point of $M$ we see that
\begin{equation}\label{E-L-I}
\wedge E^*=\wedge L^* \oplus \mathcal{I}_\wedge,
\end{equation}
 so we have an injection $W(E)\otimes\mathcal{I}_\wedge \hookrightarrow W(E)\otimes\wedge E^*$. Then we can define the left ideal
$\mathcal{I}=\mathcal{I}_L\otimes\wedge
E^* +W(E)\otimes\mathcal{I}_\wedge$ in
$W(E)\otimes\wedge E^*$ and from~(\ref{eq5}) it follows that
\begin{equation}\label{eq6}
\delta(\mathcal{I})\subset \mathcal{I}.
\end{equation}

\begin{definition}
Let $\mathbb{N}_0=\mathbb{N}\cup \{0\}$. A semigroup $(S,\vee)$ is
called \textit{filtered} if there exists a decreasing filtration $S_i,\ i\in \mathbb{N}_0$ on $S$  such that $S_0 =S$ and $S_i \vee
S_j \subset S_{i+j}\ \forall i,j$. Suppose $I=(i_1,\ldots,i_m)$ and $J=(j_i,\ldots,j_n)$ are in $\Upsilon_0$ and let $I\vee J$ be
the set $\{ i_1,\ldots,i_m, j_i,\ldots,j_n \}$ arranged in
descending order. Then $(\Upsilon_0,\vee)$ becomes a semigroup
filtered by $\Upsilon_i$.
\end{definition}

\begin{lemma}\label{lem1}
Let $\mathcal{I}^{(S)}_L$ be the ideal  in $S(E)$ generated by the
elements of $L$, then $\pi(\mathcal{I}^{(S)}_L)= \mathcal{I}_L$
under the isomorphism  of Theorem~\ref{S-W-isom}.
\end{lemma}
\begin{proof}
It is easily seen that  $(\mathcal{I}_L)_{\fm}$ [resp.
$(\mathcal{I}^{(S)}_L)_{\fm}$] is a left ideal in $W(E_{\fm})$ [resp.
in $S(E_{\fm})$] generated by the elements of $L_{\fm}$.
Since $L_{\fm}$ is isotropic w.r.t. $\omega_\fm$, we have $e_{\alpha_1}\circ
e_{\alpha_2}=e_{\alpha_2}\circ e_{\alpha_1}$
$\forall\alpha_1,\alpha_2 \in \{ 1,\ldots,\nu \}$, thus for any $I\in\Upsilon_0$ and $\forall\alpha \in \{ 1,\ldots,\nu \}$ we have
$\mu(e_I)\circ e_\alpha=\mu(e_{I\vee\{\alpha\}})$ and
$I\vee\{\alpha\}\in \Upsilon_1$. Then from Prop.~\ref{PBW}(b) it
follows that $(\mathcal{I}_L)_{\fm}  \subset\mathop{{\rm span}}_{A_{\fm}} \{ \mu_W (e_I) |\, I\in \Upsilon_1\}$. On the other hand, if $I=(i_1
,\ldots,i_p)\in \Upsilon_1$ then $1\leq i_p \leq n$, so $\mu_W
(e_I)\in(\mathcal{I}_L)_{\fm}$. Then $\mathop{{\rm span}}_{A_{\fm}} \{ \mu_W (e_I) |\, I\in \Upsilon_1\}\subset (\mathcal{I}_L)_{\fm}$  and we obtain $(\mathcal{I}_L)_{\fm}=\mu_W \iota_1 (\tilde{S}_1 (E_{\fm}))$,
where $\tilde{S}_i (E_{\fm})= \mathop{{\rm span}}_{A_{\fm}} \{ e_I |\, I\in \Upsilon_i\},\ i\in\mathbb{N}_0$ is a decreasing filtration on
$\tilde{S}(E_{\fm})$. Analogously $(\mathcal{I}_L^{(S)})_{\fm}=\mu_S \iota_1 (\tilde{S}_1 (E_{\fm}))$, which proves the lemma.
\end{proof}
From~(\ref{eq5}) it is easily seen that $\delta^{-1}$ preserves the submodule $\mathcal{I}^{(S)}_L\otimes\wedge E^*
+S(E)\otimes\mathcal{I}_\wedge$ of $S(E)\otimes\wedge E^*$,
then using Lemma~\ref{lem1}  we obtain
\begin{equation}\label{eq7}
\delta^{-1}(\mathcal{I})\subset \mathcal{I}.
\end{equation}
\begin{remark}
The choice of $\tilde{S}(E)$ in  Prop.~\ref{PBW} is crucial for our
construction of the contracting homotopy of $\delta W^\bullet$. The ordinary choice of the submodule $S'(E)$ of all symmetric tensors in $T(E)$ instead
of $\tilde{S}(E)$ yields another contracting homotopy of
$\delta W^\bullet$ which does not preserve $\mathcal{I}$.
\end{remark}
Suppose $\wp$ is a leaf of the distribution $\mathcal{D}$, $\Phi=\{ f\in A|\ f|\wp=0\}$ is the vanishing ideal of $\wp$ in $A$, $\mathcal{I}_\Phi$ is the necessarily two-sided ideal in $W(E)\otimes\wedge E^*$ generated by elements of $\Phi$, and $\mathcal{I}_{\text{fin}}=\mathcal{I}+\mathcal{I}_\Phi$ is a
homogeneous left ideal in $W(E)\otimes\wedge E^*$. Then due
to~(\ref{eq6}),(\ref{eq7}) we can define the subcomplex $\delta
\mathcal{I}_{\text{fin}}^\bullet=
(\mathcal{I}_{\text{fin}},\delta)$ with the same contracting
homotopy $\delta^{-1}$. Note that $\tau
(\mathcal{I}_{\text{fin}})=\Phi$, then using~(\ref{eq1}) we obtain
\begin{equation}\label{eq13}
H^0 (\delta\mathcal{I}_{\text{fin}}^\bullet)=\Phi, \qquad H^n
(\delta\mathcal{I}_{\text{fin}}^\bullet)=0,\ n>0
\end{equation}

\section{Connection}

Let $\nabla$ be the exterior derivative on $\wedge E^*$ which to an element $\alpha\in \wedge^{n-1} E^*$ assigns the element
\begin{equation}\label{eq8}
\begin{split}
(\nabla\alpha)(x_1,\ldots,x_n)=
\sum\limits_{1\leq i<j\leq n}
(-1)^{i+j} \alpha([x_i,x_j],x_1,\ldots,\hat{x}_i,\ldots,\hat{x}_j,\ldots x_n)
\\ +\sum\limits_{1\leq i\leq n} (-1)^{i-1}x_i \alpha
(x_1,\ldots,\hat{x}_i,\ldots, x_n).
\end{split}
 \end{equation}
Let $\nabla_x y\in E,\ x,y\in E$ be a connection on $M$, then we
can extend $\nabla_x$ to $T(E)$ by the Leibniz rule. It is well known that  a symplectic
connection preserve $\mathcal{I}_W$ for all $x\in E$, so it
induces a well-defined derivation on $W(E)$. Suppose $\{ (e_{\tilde{\alpha}},e^{\tilde{\alpha}})| \, \tilde{\alpha}=1,\ldots,\tilde{\nu} \}$ and $\{ (e_{\tilde{\beta}},e^{\tilde{\beta}})| \, \tilde{\beta}=\tilde{\nu},\ldots,\tilde{\nu}+\tilde{\nu}\vphantom{\nu}' \}$ are projective bases in $L$ and $L'$ respectively. Considering $L^*$ and ${L'}^*$ as the submodules in $E^*$ whose elements annihilate $L'$ and $L$ respectively, we see that ${L'}^* =\mathcal{I}_\wedge^1$ and $\{ (e_{\tilde{\imath}},e^{\tilde{\imath}})| \, \tilde{\imath}=1,\ldots,\tilde{\nu}+\tilde{\nu}\vphantom{\nu}' \}$ is a projective basis in $E$. Consider  the mapping~\cite{Farkas}
\begin{equation}\label{nabla-def}
 \nabla:\ W(E)\rightarrow W(E)\otimes \wedge^1 E^*, \quad
\nabla a = \sum\limits_{\tilde{\imath}=1}^{\tilde{\nu}+\tilde{\nu}\vphantom{\nu}'}
 (\nabla_{e_{\tilde{\imath}}}a) \otimes e^{\tilde{\imath}}.
\end{equation}
It is well known that $\nabla$
may be extended to a $\mathbb{R}[[\lambda]]$-linear derivation of
bidegree $(0,1)$ of the whole algebra
$W^\bullet(E)\otimes\wedge^\bullet E^*$ whose restriction to
$\wedge E^*$ coincides with~(\ref{eq8}).

We say that a polarization (or, more generally, distribution)
$\mathcal{D}$ is self-parallel w.r.t. $\nabla$ iff
\begin{equation}\label{eq9}
\nabla_x y\in L, \quad x,y\in L.
\end{equation}
For a given $\mathcal{D}$, a torsion-free connection which
obeys~(\ref{eq9}) always exists (\cite{BF}, Theorem~5.1.12).
Proceeding along the same lines as in the proof of \cite{Xu},
Lemma 5.6, we obtain a symplectic connection
 on $M$  which also obeys~(\ref{eq9}). Then from~(\ref{nabla-def}),(\ref{eq9}) it follows that $\nabla L\in \mathcal {I}$, so $\nabla\mathcal{I}_L\subset
\mathcal{I}$ since $\nabla$ is a derivative. On the other hand, the involutivity of $L$ together
with~(\ref{eq8}) yield $\nabla\mathcal{I}_\wedge
\subset\mathcal{I}_\wedge$ (Frobenius theorem), so we finally
obtain
\begin{equation}\label{eq10}
\nabla\mathcal{I}\subset\mathcal{I}.
\end{equation}

It is easily seen that the vector fields of $L$
preserve $\Phi$, i.e. $(\nabla f)(x)\in\Phi \quad \forall f\in \Phi, x\in L$, so $\nabla\Phi\in\mathcal{I}_\Phi+
\mathcal{I}^1_\wedge$ and we finally obtain
\begin{equation}\label{eq11a}
\nabla\mathcal{I}_\Phi \subset\mathcal{I}_{\text{fin}}.
\end{equation}

The following result is well known (see Theorem~3.3 of~\cite{Farkas}).
\begin{lemma}\label{gamma}
Any $A$-linear derivation of $W(E)\otimes\wedge E^*$ is quasi-inner, so there exists an element $\Gamma \in W^2 (E)\otimes \wedge^2 E^*$
such that
$$\nabla^2 a=\frac{1}{\lambda} \lbr{\Gamma}{a} \qquad \forall a\in W(E)\otimes\wedge E^* ,$$
where $\lbr{\cdot}{\cdot}$ is the commutator in $W(E)\otimes\wedge
E^*$.
\end{lemma}
We will use the expression
\begin{equation}\label{Gamma-def}
\Gamma =  \sum\limits_{\tilde{\imath}=1}^{\tilde{\nu}}
u^{-1} (e^{\tilde{\imath}}) \circ \nabla^2 e_{\tilde{\imath}} +
\sum\limits_{\tilde{\imath}=\tilde{\nu}+1}^{\tilde{\nu}+\tilde{\nu}\vphantom{\nu}'} \nabla^2 e_{\tilde{\imath}} \circ u^{-1} (e^{\tilde{\imath}})
\end{equation}
which differs from the one used~\cite{Farkas} by central terms only,
so we can use our $\Gamma$ in Lemma~\ref{gamma}.
Since $\nabla^2 L\in \mathcal {I}$ and $\mathcal{I}$ is a left ideal, we see that first term in r.h.s. of~(\ref{Gamma-def}) belongs to $\mathcal{I}$. On the other hand, $u^{-1} (e^{\tilde{\beta}}) \in L$ since $u(L)={L'}^*$ due to~(\ref{eq5}), so the second term in r.h.s. of~(\ref{Gamma-def}) belongs to $\mathcal{I}$ too, so we obtain the following proposition.
\begin{prop}\label{Gamma-lem}
An element $\Gamma \in W^2 (E)\otimes \wedge^2 E^*$
  belonging to  $\mathcal{I}$ and obeying the condition of Lemma~\ref{gamma} there exists.
\end{prop}

\section{Fedosov complex and star-product}

Let $W^{(i)}(E)$ be the grading in $W(E)$ which coincides with
$W^i (E)$ except for the $\lambda\in W^{(2)}(E)$, and let
$W_{(i)}(E)$ be the decreasing filtration generated by $W^{(i)}(E)$.
Suppose $\widehat{W}(E)$, $\widehat{\mathcal{I}}$ are completions
of $W(E)$, $\mathcal{I}$ with respect to this filtration, then
$\widehat{\mathcal{I}}$ is a left ideal in $\widehat{W}(E)\otimes
\wedge E^*$. Consider the filtration as an inverse system with natural inclusion $W_{(i+j)}(E)\subset W_{(i)}(E)$ and
let $A_i,\ i\in\mathbb{N}_0$ be the $(\lambda)$-adic
filtration in $A$, then $\tau (W_{(i)}(E))\subset A_{\{ i/2\}}$.
It is easily seen that  $\delta,\delta^{-1},\tau$ and $\nabla$ are transformations of the corresponding inverse systems, so they commute with taking inverse limits. Also it is well known that taking the inverse limits preserves short exact sequences and commutes with $\mathop{\mathrm{Hom}}(P,-)$ for any $P$. So we will write $A,W(E)$
etc. instead of $\widehat{A}$, $\widehat{W}(E)$ etc.

Let
$$r_0=0, \qquad r_{n+1}=\delta^{-1}\left(\Gamma+\nabla r_n +
\frac{1}{\lambda}r^2_n \right), \quad n\in\mathbb{N}_0 .$$ Then it
is well known that the sequence $\{r_n\}$ has a limit $r\in
W_{(2)}(E)\otimes \wedge^1 E^*$. Then we can define the well-known
Fedosov complex $DW^\bullet=(W(E)\otimes \wedge^n E^*, \ D)$ with
the differential
$$D=\delta+\nabla-\frac{1}{\lambda} \lbr{r}{\cdot}.$$
Using~(\ref{eq7}),(\ref{eq10}) and Prop.~\ref{Gamma-lem} and
taking into account that $\mathcal{I}$ is a left ideal in
$W(E)\otimes\wedge E^*$ we have $r_n\in \mathcal{I}$ for all $n$,
so $r\in \mathcal{I}$.
Using~(\ref{eq6}),(\ref{eq7}),(\ref{eq10}),(\ref{eq11a})  we see that $D\mathcal{I}_{\text{fin}} \subset \mathcal{I}_{\text{fin}}$
and $Q\mathcal{I}_{\text{fin}} \subset \mathcal{I}_{\text{fin}}$,
so we can define the subcomplex
$D{\mathcal{I}}_{\text{fin}}^\bullet=({\mathcal{I}}_{\text{fin}},D)$.
Define the left $W(E)\otimes\wedge E^*$-module $F=W(E)\otimes\wedge E^* /\mathcal{I}_{\mathrm{fin}}$ with the grading induced from $W(E)\otimes\wedge^\bullet E^*$, then we can define factor-complexes $\delta F^\bullet =(F^n,\delta)$ and $DF^\bullet=(F^n, D)$.

\begin{lemma}[\cite{Donin}]\label{lem3}
Let $F$ be an Abelian group which is complete with respect to its
decreasing filtration $F_i ,\ i\in\mathbb{N}_0$ such that $\cup F_i=F$ and $\cap F_i=\emptyset$. Let $\deg a=\max \{i: a\in F_i\}$ for $a\in
F$ and let $\varphi:\ F\rightarrow F$ be a set-theoretic map such that
$\deg (\varphi(a)-\varphi(b))>\deg (a-b)$ for all $a,b\in F$. Then the map $Id+\varphi$ is invertible.
\end{lemma}
Let $Q:\ W(E)\otimes\wedge E^* \rightarrow W(E)\otimes\wedge E^*$,
be the $\mathbb{R}[[\lambda]]$-linear
map $Q=Id+\delta^{-1}(D-\delta)$, then it is well known that $\delta Q=QD$ and from
Lemma~\ref{lem3} it follows that $Q$ yield an isomorphism in cohomology. Since $Q\mathcal{I}_{\mathrm{fin}}\subset \mathcal{I}_{\mathrm{fin}}$, we obtain the following commutative diagram of complexes with exact rows:
$$\xymatrix{
0 \ar[r] & \delta \mathcal{I}^\bullet_{\mathrm{fin}} \ar[r] & \delta W^\bullet \ar[r] & \delta F^\bullet \ar[r] & 0 \\
0 \ar[r] &  D\mathcal{I}^\bullet_{\mathrm{fin}} \ar[r] \ar[u]_{H(Q)} & DW^\bullet \ar[r] \ar[u]_{H(Q)} & DF^\bullet \ar[r] \ar[u]_\cong & 0.
}$$
Using~(\ref{eq4}),(\ref{eq13}) and the long exact sequence, we obtain
\begin{equation}\label{eq14}
H^0 (\delta F^\bullet)= A/\Phi, \qquad H^n
(\delta F^\bullet)=0,\ n>0.
\end{equation}
Then we can carry the structure of $\mathbb{R}$-algebra  from $H^0 (DW^\bullet)$  to $H^0 (\delta W^\bullet)$ and convert
 the structure of  left $H^0 (DW^\bullet)$-module on  $H^0 (DF^\bullet)$ into the structure of left $H^0 (\delta W^\bullet)$-module on $H^0 (\delta F^\bullet)$.
  Due to~(\ref{eq4}),(\ref{eq14}) this gives the Fedosov-type star-product $\ast_L$ on $A$ and the structure of a left $(A,\ast_L)$-module on $A/\Phi\cong C^\infty (\wp)$, so we obtain the following theorem.
\begin{theorem}
Let $M$ be a symplectic manifold and let $\mathcal{D}\subset TM$ be a real polarization on $M$.  Then  there exists a star-product $\ast_L$ on $M$  such that for an arbitrary leaf $\wp$  of $\mathcal{D}$ the $\mathbb{R}$-algebra $C^\infty (\wp)$ has a natural structure of a left $(C^\infty (M),\ast_L)$-module.
 \end{theorem}

For the realization of $\ast_L$ in local charts using bidifferential operators see~\cite{IJGMP}.


\begin{thebibliography}{99}

\bibitem{Ber74} F.A. Berezin,  Quantization,  \textit{Math. USSR, Izv.} \textbf{8}  (1974), 1109-1165

\bibitem{BFFLS} F. Bayen, M. Flato, C. Fronsdal, A. Lichnerowicz and D. Sternheimer, Deformation theory and quantization.
\emph{Ann. Phys.} \textbf{111}  (1978), 61-110, 111-151

\bibitem{Hueb} J. Huebschmann,
Poisson cohomology and quantization, \textit{J. Reine Angew. Math.} \textbf{408} (1990), 57-113

\bibitem{Grab} J. Grabowski,
Abstract Jacobi and Poisson structures. Quantization and star-products, \emph{J. Geom. Phys.} \textbf{9} (1992), 45-73

\bibitem{Fedosov1} B.V. Fedosov,  A simple geometrical construction of
deformation quantization, \textit{J. Diff. Geom.} \textbf{40} (1994), 213-238

\bibitem{Fedosov2} B.V. Fedosov, \textit{Deformation quantization and index theory}, Akademie, B., 1996

\bibitem{Donin} J. Donin,  On the quantization of Poisson brackets,
\textit{Adv. Math.} \textbf{127} (1997), 73-93

\bibitem{Farkas} D. Farkas,  A ring-theorist's description of Fedosov quantization, \textit{Lett. Math. Phys.} \textbf{51} (2000), 161-177

\bibitem{WaldStates} S.~Waldmann, States and representations in deformation quantization, \textit{Rev. Math. Phys.} \textbf{17} (2005), 15-75

\bibitem{BordTrav}    M.~Bordemann, (Bi)modules, morphisms, and reduction of star-products: the symplectic case, foliations and obstructions, \textit{Travaux Math.} \textbf{16} (2005), 9-40

\bibitem{IJGMP} S.A. Pol'shin,  A cohomological construction of modules over Fedosov deformation quantization algebra, \textit{Int. J. Geom. Meth. Mod. Phys.} \textbf{5} (2008) 547-556

\bibitem{BGHHW}
M. Bordemann, G. Ginot, G. Halbout, H.-C. Herbig, and S. Waldmann,
Formalit\'{e} $G_\infty$ adapt\'{e}e et star-repr\'{e}sentations
sur  des sous-vari\'{e}t\'{e}s co\"{\i}sotropes,
\url{math.QA/0504276}

\bibitem{Morye} A.S. Morye, Note on the Serre-Swan Theorem, \url{arXiv:0905.0319}

\bibitem{Vaser} L.N. Vaserstein, Vector bundles and projective modules, \textit{Trans. Amer. Math. Soc.} \textbf{294} (1986), 749-755

\bibitem{Vin} J. Nestruev, \textit{Smooth manifolds and observables}, Springer, NY, 2003


\bibitem{Rin} G. Rinehart, Differential forms for general commutative algebras, \textit{Trans. Amer. Math. Soc.}, {\bf 108} (1963), 195-222.

\bibitem{Etayo} F.Etayo Gordejuela  and R.Santamaria,  The canonical connection on a bi-lagrangian manifold, \textit{J. Phys. A: Math.Gen.} \textbf{34} (2001), 981-987

\bibitem{HKR} G. Hochschild, B. Kostant and A. Rosenberg,
Differential forms on regular affine algebras,
\textit{Trans. Amer. Math. Soc.} \textbf{102} (1962), 383-408

\bibitem{Sridharan} R. Sridharan,  Filtered algebras and  representations of Lie algebras, \textit{Trans. Amer. Math. Soc.} \textbf{100} (1961), 530-550

\bibitem{BF} A.Bejancu and H.R.Farran,  \textit{Foliations and geometric structures}, Springer, Dordrecht, 2006


\bibitem{Xu} P.Xu,  Fedosov $\ast$-products and quantum momentum maps,
\textit{Commun. Math. Phys.} \textbf{197} (1998), 167-197




\end{thebibliography}
\end{document}